\documentclass[11pt]{article}
\usepackage[figuresright]{rotating}
\usepackage{fullpage}
\usepackage{amssymb,amsmath}
\usepackage[utf8x]{inputenc}
\usepackage[T1]{fontenc}
\usepackage{siunitx}
\usepackage[version=3]{mhchem}
\usepackage{authblk}
\usepackage{epsfig}
\usepackage{amsfonts}
\usepackage{amsthm}
\allowdisplaybreaks
\usepackage{latexsym}
\usepackage{bm}
\usepackage{tabularx}
\usepackage{booktabs} 
\usepackage{enumerate}
\usepackage{caption}
\usepackage[titletoc]{appendix}
\usepackage[tight,footnotesize]{subfigure}
\usepackage{mathtools}
\usepackage[dvipsnames]{xcolor}
\usepackage{setspace}
\theoremstyle{plain}
\newtheorem{theorem}{Theorem}[section]
\newtheorem{lemma}[theorem]{Lemma}
\newtheorem{corollary}[theorem]{Corollary}
\newtheorem{proposition}[theorem]{Proposition}

\title{\huge Overlapping time of a virtual customer in time-varying many-server queues}

\author[1,2]{Young Myoung Ko \thanks{Email: youngko@postech.ac.kr}}
\author[3]{Jin Xu \thanks{Corresponding author. Email: xu\_jin@hust.edu.cn}}
\affil[1]{Department of Industrial and Management Engineering, Pohang University of Science and Technology, 77, Cheongam-ro, Nam-gu, Pohang, Gyeongbuk, 37673, Republic of Korea}
\affil[2]{Open Innovation Big Data Center, Pohang University of Science and Technology, 77, Cheongam-ro, Nam-gu, Pohang, Gyeongbuk, 37673, Republic of Korea.}
\affil[3]{School of Management, Huazhong University of Science and Technology, Wuhan 430074, Hubei, China}
\date{}

\begin{document}
\maketitle
\begin{abstract}
 Motivated by the ongoing COVID-19 pandemic, this paper investigates customers' infection risk by evaluating the overlapping time of a virtual customer with others in queueing systems. Most of the current methodologies focus on characterizing the risk in stationary systems, which may not apply to the more practical time-varying systems. As such, we propose an approximation framework that relies on the fluid limit to compute the expected overlapping time in time-varying queueing systems. Simulation experiments verify the accuracy of our approach.
\end{abstract}


\section{Introduction}
The service system is one of the highly risky areas during pandemic seasons where customers are in close contact with each other. Measuring the infection risk of customers in service systems is crucial for providing guidance to service providers on system design and operations. However, computing the infection risk in service systems is challenging due to many systems having a time-varying nature. For instance, the arrival process of customers into grocery stores and hospitals is usually non-stationary or cyclic due to peak and off-peak hours. 

In this research, we aim to measure the customers' infection risk in a multi-server time-varying queueing system by considering the \emph{overlapping time} of a virtual customer arriving at an arbitrary time. Overlapping time measures the time that the virtual customer would overlap with other customers during her sojourn time, thus is an effective metric to evaluate the likelihood that customers get infected while waiting in service systems  \cite{Kang2022, Perlman2020, Perlman2021, Palomo2021, Pender2021}. The longer a customer overlaps with other customers who may carry the virus, the more likely that she will get infected (see Sun and Zhai \cite{Sun2020}). 

Most of the current studies consider the customers' infection risk in stationary systems. For instance, Kang et al. \cite{Kang2022} considered the \textit{system-specific basic reproduction rate} in stationary Markovian queues. Perlman and Yechiali \cite{Perlman2020} used the factorial moment of the queue length to measure the risk for customers in shopping and cashier areas with a stationary queueing setting.  Perlman and Yechiali \cite{Perlman2021} further investigated the infection risk based on customers' decisions to join or balk in Markovian queues. 
Palomo and Pender \cite{Palomo2021} derived the steady-state distribution of the overlapping time for the \(M/M/1\) queue and obtained new conditional moments provided waiting or overlapping time of a customer. The methodologies used in these studies mainly work on characterizing the steady-state behavior of stationary queueing systems. 

Pender and Palomo \cite{Pender2021} is the most relevant work to our research, where they derived the number of overlaps for the infinite-server queue with time-varying arrivals. Instead, we consider a more realistic but challenging scenario with a finite number of servers, where the exact performance analysis is usually intractable when the system is time-varying.

To deal with the challenges in analyzing the overlapping time in time-varying multi-server systems, we propose an approximation framework that relies on the \textit{fluid limit}. However, the standard fluid limit approximation turns out to be inaccurate when the system is near the critically loaded point \cite{Ko2013, Massey2013}. We use the \textit{adjusted fluid model} (see \cite{Ko2013}) to tackle this issue, and the adjusted model significantly improves the approximation quality. We extend our discussion to the infinite-server system and derive the \textit{closed-form expression} of the overlapping time when the arrival rate has a cyclic form. Numerical experiments verify the accuracy of the analytical solution. Our approximation framework can be used to evaluate the infection risk of customers arriving at different times in a service system, which can further assist service providers in adjusting their operating hours and making disinfecting plans.

The remainder of this paper is organized as follows. Section~\ref{sec:prob} explains the problem of interest. Section~\ref{sec:multiserver} proposes an approximation method for estimating the mean overlapping time using the (adjusted) fluid limit. Section~\ref{sec:infinite} derives a closed-form expression for the overlapping time in an infinite-server setting. Section~\ref{sec:conclusion} wraps up the paper and discusses possible future extensions.

\section{Problem description} \label{sec:prob}
Motivated by a grocery store with multiple cashiers, we consider a time-varying multi-server queue. The arrival process of customers follows a non-homogeneous Poisson process with rate \(\lambda(t)\), and the service times follow an exponential distribution with parameter \(\mu\). The number of servers (cashiers) is \(n\). So our system model is an \(M_t/M/n\) queue. 
We aim to estimate the \textit{overlapping time} of a virtual customer who arrives at the system at time $t$. The overlapping time is defined as the total time this virtual customer would overlap with the other customers if she enters the queue, during her sojourn time (from arrival to departure) in the system. For example, suppose a virtual customer spends three minutes in the system, and two other customers are present during her stay. Then, the overlapping time is \(3 \text{ minutes} \times 2 \text{ customers} = 6 \text{ minutes}\). 

\section{Overlapping time for \(M_t/M/n\) queues}\label{sec:multiserver}
Let \(X(t)\) denote the number of customers in the system at time \(t\). Then, \(X(t)\) is the solution to the following integral equation:
\begin{equation}
X(t) = X(0) + Y_1\left(\int_{0}^{t} \lambda(s)ds\right) - Y_2\left(\int_{0}^{t} \mu \min(X(s),n) ds\right).\label{eqn:sample_path_mmn}
\end{equation}

The terms \(Y_1(\cdot)\) and \(Y_2(\cdot)\) are independent unit-rate Poisson processes, and they count the number of arrivals and departures until time \(t\), respectively. We note that this representation is widely used in queueing literature to describe the sample path of a time-varying Markovian queue, e.g., \cite{Kurtz1978, Mandelbaum1998, Mandelbaum2002, Ko2010, Ko2013, Pender2017}.

When we have Lipschitz rate functions \(f_i(t,x)\) and define \(F_i(t,x)=\int_0^t f_i(s,x)ds\), then the composition \(Y_i\circ F_i\) forms a non-homogeneous Poisson process with parameter \(f_i(t,x)\). In equation~\eqref{eqn:sample_path_mmn}, we have \(f_1(t,x) = \lambda(t)\) and \(f_2(t,x) = \mu \min\left(x,n\right)\); both are Lipschitz with respect to \(x\).

Suppose the virtual customer \textbf{C} arrives at time \(\tau\). We want to calculate the \textit{overlapping time} of \textbf{C} with other customers in the system during \textbf{C}'s stay. To do so, we define a new process \(\{Z(\tau, t), t\ge 0\}\) as the number of customers that customer \textbf{C} met in the system at time $\tau$ and are still present in the system at time $\tau + t$. Therefore, for $t\geq 0$, we have:
\begin{equation}
Z(\tau, t) = X(\tau) - Y_2\left(\int_{\tau}^{\tau+t}\mu \min(Z(\tau,s),n)ds\right).\label{eqn:z_mmn}
\end{equation}

We define \(t_0\) to be the first time that \(Z(\tau,t_0)\) becomes less than or equal to \(n-1\):
\begin{align*}
    t_0 = \inf\{t \ge 0: Z(\tau, t) \le n-1\}.
\end{align*}

Let \(S_{\tau}\) and \(O_{\tau}\) be the service time and overlapping time of the customer arriving at time \(\tau\), respectively. Then, this customer overlaps with other customers for at least \(t_0\) and at most \(t_0 + S_{\tau}\) amount of time. The total overlapping time \(O_{\tau}\) can be calculated by integrating \(X(t)\) over the overlapping period, and it satisfies
\[\int_{\tau}^{\tau+t_0} X(t) dt \le O_{\tau} \le \int_{\tau}^{\tau+t_0+S_{\tau}} X(t) dt. \]

The idle time of the system (or the idle probability) under a heavy-traffic condition tends to be short (or small). In such cases, we approximate the expected overlapping time as
\begin{align}
E\left[O_{\tau}\right] &\approx E\left[\int_{\tau}^{\tau+t_0+S_{\tau}} X(t) dt\right].\label{eqn:expected_overlap_approx_mmn}
\end{align}

Calculating the expected overlapping time in \eqref{eqn:expected_overlap_approx_mmn} is intractable in general. What we will do next is to approximate the RHS of \eqref{eqn:expected_overlap_approx_mmn} using the fluid limit. 

\subsection{Fluid limit}
Consider a sequence of stochastic processes, \(\{X^{\eta}(t), t\ge 0\}\) that satisfies the following integral equation:
\begin{align}
X^{\eta}(t) &= X^{\eta}(0) + Y_1\left(\int_{0}^{t}\eta \lambda(s)ds\right)\nonumber \\
&\quad- Y_2\left(\int_{0}^{t} \mu \min(X^{\eta}(s),\eta n) ds\right).\label{eqn:sequence_mmn}
\end{align}
In Equation~\eqref{eqn:sequence_mmn}, \(X^{\eta}(t)\) accelerates the arrival rate and the number of servers by \(\eta\). By taking \(\eta \rightarrow \infty\), we obtain the fluid limit as follows:

\begin{lemma}
\[\lim_{\eta\rightarrow\infty}\frac{X^{\eta}(t)}{\eta} = x(t), \text{ a.s. for } t\in [0, T],\]
where \(T>0\) and \(x(t)\) is the solution to the following ODE:
\[\frac{d}{dt}x(t) = \lambda(t) - \mu \min\left(x(t),n\right) \text{ with } x(0)=\lim_{\eta\rightarrow\infty}\frac{X^{\eta}(0)}{\eta}.\]
\begin{proof}
The result follows directly from Kurtz \cite{Kurtz1978} and Mandelbaum et al. \cite{Mandelbaum1998}.
\end{proof}
\end{lemma}
So, for a sufficiently large \(\eta\), we can approximate \(X^{\eta}(t)\) as
\[X^{\eta}(t) \approx \eta x(t), \text{ for } t \in [0,T].\]

To compute the overlapping time of a virtual customer arriving at time \(\tau\), we consider a sequence of stochastic processes, \(\{Z^{\eta}(\tau,t), t\ge 0\}\),  satisfying
\begin{align*}
    Z^{\eta}(\tau, t) = X^{\eta}(\tau) - Y_2\left(\int_{\tau}^{\tau+t} \mu \min(Z^{\eta}(\tau,s),\eta n)ds\right).
\end{align*}
Then, we have the fluid limit as follows:
\begin{proposition}
\[\lim_{\eta \rightarrow \infty} \frac{Z^{\eta}(\tau,t)}{\eta} = z(\tau, t) \text{  a.s. for } t \in [0,T],\]
where \(T>0\) and \(z(\tau,t)\) is the solution to the following ODE:
\[\frac{d}{dt}z(\tau,t) = -\mu \min(z(\tau,t),n)\text{ provided } z(\tau,0)=x(\tau).\]
\begin{proof}
The result follows from Kurtz \cite{Kurtz1978} and Mandelbaum et al. \cite{Mandelbaum1998}.
\end{proof}
\end{proposition}

Let \(t_0^{\eta}\) be the first time that \(Z^{\eta}(\tau,t)\) gets less than or equal to \(\eta (n-1)\):
\begin{equation}
\begin{aligned}\label{eqn:t_0_eta}
    t_0^{\eta} = \inf\{t \ge 0: Z^{\eta}(\tau, t) \le \eta (n-1)\}.
\end{aligned}
\end{equation}

Then, we have the following corollary: 
\begin{corollary}
\[\lim_{\eta\rightarrow \infty} t_0^{\eta} = t_0 \text{  a.s.,}\]
where 
\begin{equation*}
t_0 = \inf\{t\ge0: z(\tau,t)\le n-1\}
\end{equation*}
\begin{proof}
This is a special case of the result in Ethier and Kurtz \cite{Kurtz86}.
\end{proof}
\end{corollary}
\begin{theorem}\label{thm:conv_as}
For \(t, \tau, \tau+ t_0 \in [0,T]\),
\begin{align}
   &\lim_{\eta\rightarrow \infty} \int_{\tau}^{\tau+t_0^{\eta}}\frac{X^{\eta}(t)}{\eta}dt = \int_{\tau}^{\tau+t_0} x(t) dt \text{  a.s.} \label{eqn:conv_int_x}\\
   &\lim_{\eta\rightarrow \infty} \int_{0}^{t_0^{\eta}}\frac{Z^{\eta}(\tau,t)}{\eta}dt = \int_{0}^{t_0} z(\tau,t) dt \text{  a.s.}\label{eqn:conv_int_z}
\end{align}
\begin{proof}
We define 
\[h_{\tau}(x,t) = \int_{\tau}^{\tau+ t}x ds.\]
Then, \(h_{\tau}\) is a continuous function in \((x,t)\). 
 Therefore, by applying the continuous mapping theorem (Chen and Yao \cite{Chen01}), we have 
 \begin{align*}
\int_{\tau}^{\tau+t_0^{\eta}}\frac{X^{\eta}(t)}{\eta} dt &= h_{\tau}\left(\frac{X^{\eta}(t)}{\eta}, t_0^{\eta}\right)\\
&\rightarrow h_{\tau}(x(t), t_0) = \int_{\tau}^{\tau+t_0} x(t) dt \text{  a.s.} 
\end{align*}
\end{proof}
\end{theorem}

Once we have the fluid limit, we then approximate the expected overlapping time by integrating the fluid limit from \(\tau\) to \(\tau+t_0+1/\mu\), so that
\[E[O_{\tau}] \approx \int_{\tau}^{\tau + t_0 + 1/\mu} x(t) dt.\]

The approximation quality of the overlapping time will heavily rely on the approximation accuracy of the fluid limit $x(t)$. Figure~\ref{fig:inaccurate_fluid} compares the fluid limit and simulation when there are 30 servers. The arrival and service rates are 
\(\lambda(t) = 30(0.5 \sin(0.5 t)+1.0)\) and \(\mu = 1.0\), respectively. For simulation, we average independent 10,000 simulation runs. As shown in Figure~\ref{fig:inaccurate_fluid}, the fluid limit does not accurately approximate the mean behavior of the system after it crosses the critically loaded point, i.e., \(n = 30\).

\begin{figure}[htbp]
	\begin{center}
		\includegraphics[width = .45\textwidth]{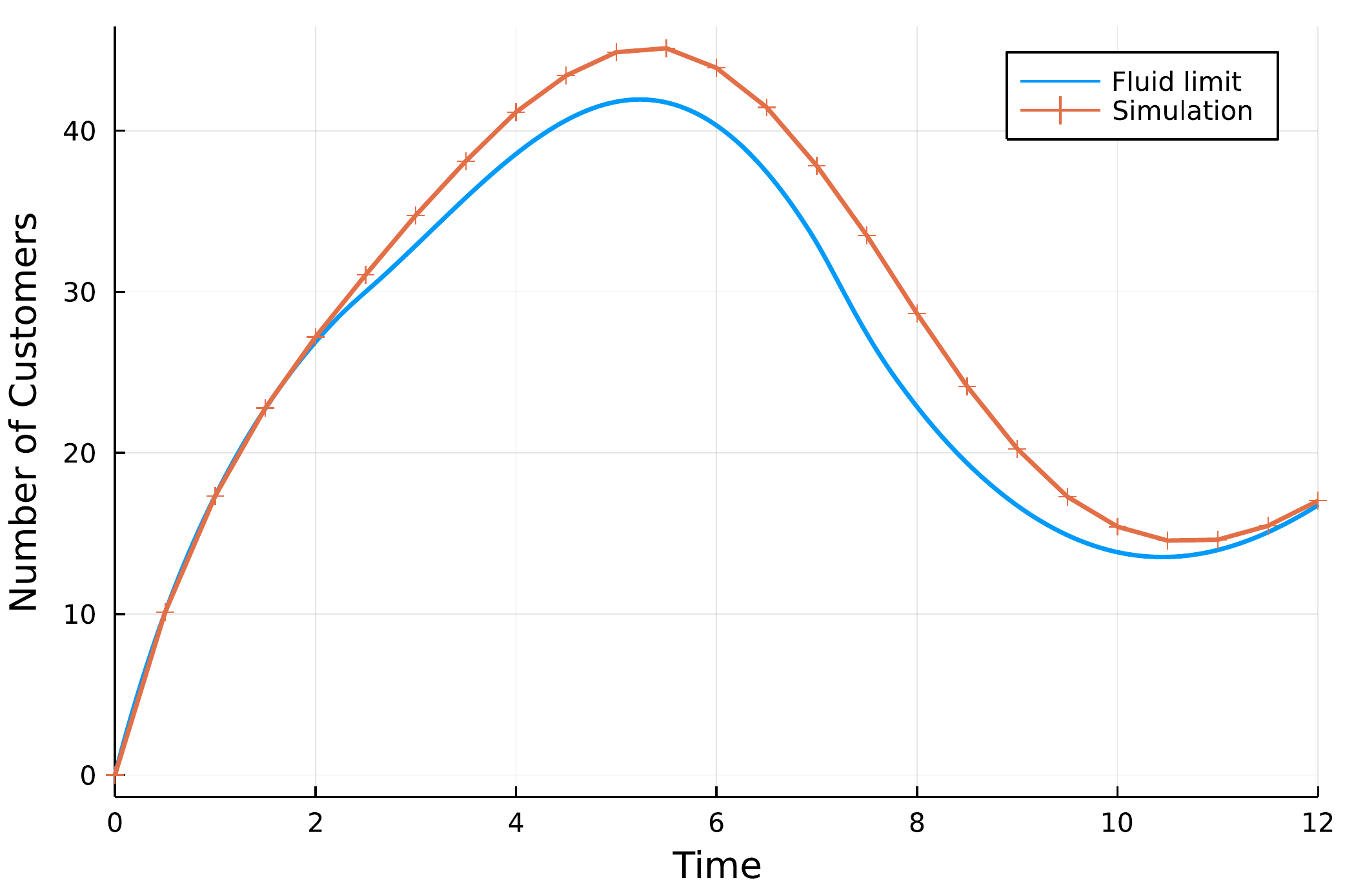}
    \end{center}
\caption{Fluid limit vs. Simulation} \label{fig:inaccurate_fluid}
\end{figure}

\subsection{Adjusted Fluid Model}

Ko and Gautam \cite{Ko2013} pointed out that the fluid limit fails to approximate the system dynamics when the system state is lingering around the critical number---the number of servers \(n\). The problem arises from the nondifferentiability of the rate function \(\min(x,n)\). To overcome this problem, Ko and Gautam \cite{Ko2013} proposed the adjusted fluid and diffusion models that significantly improve approximation accuracy for the dynamics of large-but-finite-sized systems. This result also corresponds to the Gaussian approximation in Massey and Pender \cite{Massey2013}.

In our problem, the accuracy of the fluid approximation is crucial because it directly affects the accuracy of its integration. So, instead of using rate functions \(f_1(t,x) = \lambda(t)\) and \(f_2(t,x) = \mu \min(x,n)\), we use the rate functions for the adjusted fluid model in Ko and Gautam \cite{Ko2013} as follows:
\begin{align*}
g_1(t,x,u) &= \lambda(t) \\    
g_2(t,x,u) &= \mu\left[n + (x-n)\Phi(n,x,\sqrt{u}) - u\phi(n,x,\sqrt{u})\right],
\end{align*}
where \(\Phi(a,b,c)\) and \(\phi(a,b,c)\) are the cumulative distribution function and the probability density function values, respectively, at point \(a\) of the Gaussian distribution with mean \(b\) and standard deviation \(c\).

We consider a sequence of stochastic processes, \(\{X_a^{\eta}(t), t\ge 0\}\), satisfying the following integral equation:
\begin{equation}
\begin{aligned}\label{eqn:sequence_adj_mmn} 
X_a^{\eta}(t) &= X_a^{\eta}(0) + Y_1\left(\int_{0}^{t}\eta g_1(s,X_a^{\eta}(s),Var[X_a^{\eta}(s)])ds\right) \\
&- Y_2\left(\int_{0}^{t}\eta g_2(s,X_a^{\eta}(s),Var[X_a^{\eta}(s))] ds\right).
\end{aligned}
\end{equation}
As seen in \eqref{eqn:sequence_adj_mmn}, the sequence involves the variance of \(X_a^{\eta}(t)\) related to the diffusion limit. Following Ko and Gautam \cite{Ko2013}, we have the adjusted fluid and diffusion limits in the following proposition.

\begin{proposition}[Adjusted fluid and diffusion limits for \(X_a^{\eta}(t)\)]
\begin{align*}
&\lim_{\eta\rightarrow\infty}\frac{X_a^{\eta}(t)}{\eta} = x_a(t)\text{ a.s.}\\
&\lim_{\eta\rightarrow\infty}\sqrt{\eta}\left(\frac{X_a^{\eta}(t)}{\eta} - x_a(t)\right) \overset{d}{=} D(t)\text{ for } t\in[0,T].
\end{align*}
where \(T>0\) and \(D(t)\) is a Gaussian process with zero mean and variance \(u(t)\). We can obtain \(x_a(t)\) and \(u(t)\) by solving the following ODEs \cite{Arnold92}:
\begin{align*}
    &\frac{d}{dt}x_a(t) = g_1\left(t,x_a(t), u(t)\right) - g_2\left(t,x_a(t), u(t)\right)\\
    &\frac{d}{dt}u(t) = -2\mu\Phi\left(n,x_a(t),\sqrt{u(t)}\right)u(t) + g_1\left(t,x_a(t),u(t)\right)\\
    &\quad\quad\quad+ g_2\left(t,x_a(t), u(t)\right).
\end{align*}
\begin{proof}
It is an application of the result in Ko and Gautam \cite{Ko2013}.
\end{proof}
\end{proposition}

Figure~\ref{fig:accurate_adjusted} illustrates the result of the adjusted limits with the same setting in Figure~\ref{fig:inaccurate_fluid}. We easily notice that the accuracy in approximating the number of customers significantly improves. 

We now show how the adjusted limits can be used to compute the overlapping time. When a customer arrives at time \(\tau\), we consider a sequence of stochastic process, \(\{Z_a^{\eta}(\tau, t), t\ge 0\}\) satisfying the following integral equation:
\begin{equation*}
Z_a^{\eta}(\tau, t) = X_a^{\eta}(\tau) - Y_2\left(\int_{0}^{t}\eta g_2(s,X_a^{\eta}(s),Var[X_a^{\eta}(s))] ds\right).
\end{equation*}

Then, we have the adjusted fluid and diffusion limits for \(\{Z_a^{\eta}(\tau,t), t\ge 0\}\) as follows:
\begin{proposition}[Adjusted limits for \(Z_a^{\eta}(\tau,t)\)]
\begin{align*}
&\lim_{\eta\rightarrow\infty}\frac{Z_a^{\eta}(\tau, t)}{\eta} = z_a(\tau, t)\text{ a.s.}\\
&\lim_{\eta\rightarrow\infty}\sqrt{\eta}\left(\frac{Z_a^{\eta}(\tau, t)}{\eta} - z_a(\tau, t)\right) \overset{d}{=} E(t)\text{ for } t\in[0,T].
\end{align*}
where \(T>0\) and \(E(t)\) is a Gaussian process with zero mean and variance \(v(t)\). We can obtain \(z_a(\tau, t)\) and \(v(t)\) by solving the following ODEs:
\begin{align*}
    &\frac{d}{dt}z_a(\tau, t) =  - g_2\left(t,z_a(\tau, t), v(t)\right)\\
    &\frac{d}{dt}v(t) = -2\mu\Phi\left(n,z_a(\tau, t),\sqrt{v(t)}\right)v(t) \\
    & + g_2\left(t,z_a(\tau, t), v(t)\right).
\end{align*}
\begin{proof}
This is an application of the result in Ko and Gautam \cite{Ko2013}.
\end{proof}
\end{proposition}

\begin{figure}[t]
	\begin{center}
		\includegraphics[width = .45\textwidth]{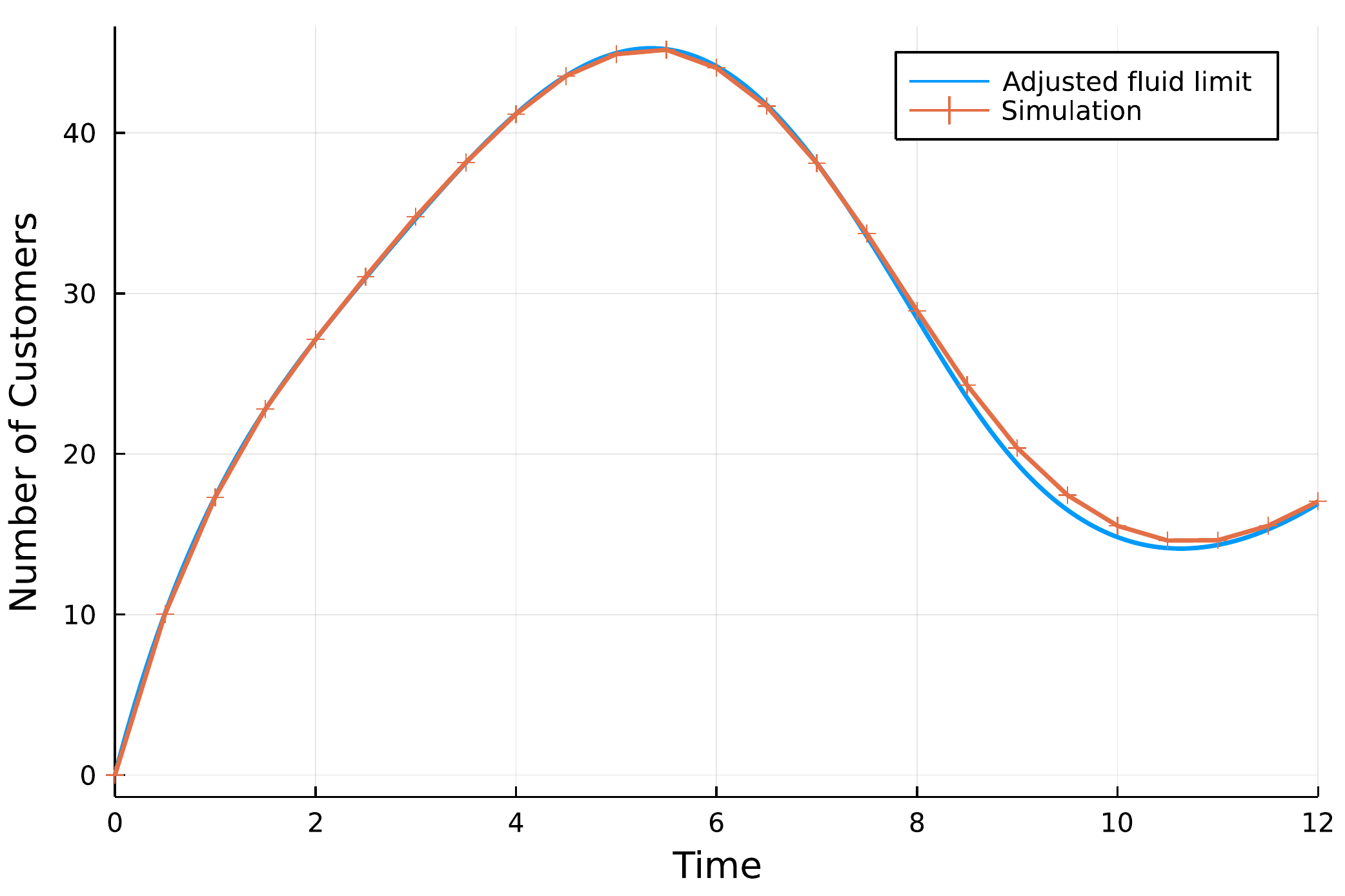}
\end{center}
\caption{Adjusted fluid limit vs. simulation} \label{fig:accurate_adjusted}
\end{figure}

 Let \(t_a^{\eta}\) be the first time that \(Z_a^{\eta}(\tau,t)\) becomes less than or equal to $\eta(n-1)$. Similar to the definition of \(t_0^{\eta}\), we have
\begin{align*}
    t_a^{\eta} = \inf\{t \ge 0: Z_a^{\eta}(\tau, t)\le \eta(n-1)\}.
\end{align*}

In the following corollary, we show that \(t_a^{\eta}\) converges.
\begin{corollary}
\[\lim_{\eta\rightarrow \infty} t_a^{\eta} = t_a \text{  a.s.,}\]
where
\begin{equation*}
t_a = \inf\{t\ge 0: z_a(\tau,t)\le n-1\}.
\end{equation*}
\end{corollary}
\begin{proof}
We have the corollary by applying the results in Ethier and Kurtz \cite{Kurtz86} to Ko and Gautam \cite{Ko2013}.
\end{proof}
For the adjusted fluid model, we have a similar result in Theorem~\ref{thm:conv_as} as follows:
\begin{theorem}\label{thm:inte_conv}
For \(t, \tau, \tau+ t_a \in (0,T]\),
\begin{align*}
   &\lim_{\eta\rightarrow \infty} \int_{\tau}^{\tau+t_a^{\eta}}\frac{X_a^{\eta}(t)}{\eta}dt = \int_{\tau}^{\tau+t_a} x_a(t) dt\text{  a.s.} \\
   &\lim_{\eta\rightarrow \infty} \int_{0}^{t_a^{\eta}}\frac{Z_a^{\eta}(\tau,t)}{\eta}dt = \int_{0}^{t_a} z_a(\tau,t) dt \text{  a.s.}
\end{align*}
\begin{proof}
Similar to the proof of Theorem~\ref{thm:conv_as}, the continuous mapping theorem (Chen and Yao \cite{Chen01}) gives the theorem.
\end{proof}
\end{theorem}

Following from Theorem \ref{thm:inte_conv}, we can then approximate the expected overlapping time using the adjusted fluid model as follows:
\begin{equation}
E[O_{\tau}] \approx \int_{\tau}^{\tau + t_a + 1/\mu} x_a(t) dt.
\end{equation}

\begin{table*}[h!]
\centering
\caption{Overlapping times: fluid vs adjusted vs simulation}\label{tbl:anal_vs_sim_mmn}
\begin{tabular}{c|c|c|c|c|c|c}
\hline
\(\alpha\) & \(\beta\) & \(\rho\) & \(\tau\) & Fluid (error) & Adjusted (error) & Simulation (variance) \\ \hline
0.5 & 0.3 & 0.8 & 3 & 29.46 (15.1\%) & 31.52 (9.2\%) & 34.71 (34.60) \\
0.5 & 0.3 & 0.8 & 6 & 26.19 (28.0\%) & 34.69 (4.6\%) & 36.37 (32.78) \\
0.5 & 0.5 & 0.8 & 3 & 41.02 (17.0\%) & 46.54 (5.8\%) & 49.40 (47.23) \\
0.5 & 0.5 & 0.8 & 6 & 48.64 (19.5\%) & 61.19 (1.2\%) & 60.45 (46.31) \\
0.5 & 0.5 & 0.8 & 9 & 15.02 (18.7\%) & 16.71 (9.5\%) & 18.47 (21.10) \\
1.0 & 0.3 & 0.8 & 9 & 27.19 (15.8\%) & 31.44 (2.6\%) & 32.29 (28.43) \\
1.0 & 0.5 & 0.8 & 3 & 28.24 (15.6\%) & 33.66 (0.6\%) & 33.45 (28.11) \\
1.0 & 0.5 & 0.8 & 9 & 36.13 (11.6\%) & 42.98 (5.2\%) & 40.85 (31.76) \\
\hline
\end{tabular}
\end{table*}

\subsection{Numerical examples} \label{sec:numerical}
We conduct numerical experiments to see how accurately the adjusted fluid limits  approximate overlapping times. We consider an \(M_t/M/n\) queue with the arrival process having a rate \(\lambda(t)\) function as: 
\begin{equation}
\lambda(t) = \left(\beta \sin(\alpha t) + \lambda\right)\cdot n \cdot \rho.\label{eqn:lambda_t_mmn}
\end{equation}
The parameter \(\beta\) determines the amplitude of the arrival rate change, \(\alpha\) determines the frequency, \(\lambda\) denotes the baseline arrival rate, \(n\) represents the number of servers, and \(\rho\) designates overall mean traffic intensity. 


Table~\ref{tbl:anal_vs_sim_mmn} provides a comparison of the fluid model, the adjusted model, and the simulation with the rate function (\ref{eqn:lambda_t_mmn}) having $n=30$, $\mu=1.0$, and $\lambda=1.0$. We conduct 10,000 independent simulation runs for comparison. As we see from Table~\ref{tbl:anal_vs_sim_mmn}, the inaccuracy of the fluid model in approximating the number of customers also propagates to the estimation of overlapping times.
In contrast, the adjusted model performs quite well in approximating the overlapping times.


Figure~\ref{fig:histogram} plots the histogram of the overlapping times at different arrival times (\(\tau\)) under the same setting (\(\alpha = 0.5, \beta=0.5, \rho=0.8, n=30\)) as Figures~\ref{fig:inaccurate_fluid} and \ref{fig:accurate_adjusted}. Though we do not derive the distribution of the overlapping times for time-varying multi-server queues, the distributions from the simulation shows an exponential shape.

\begin{figure*}[t!]
	\begin{center}
		\subfigure[\(\tau = 3\)] {
		\includegraphics[width = .4\textwidth]{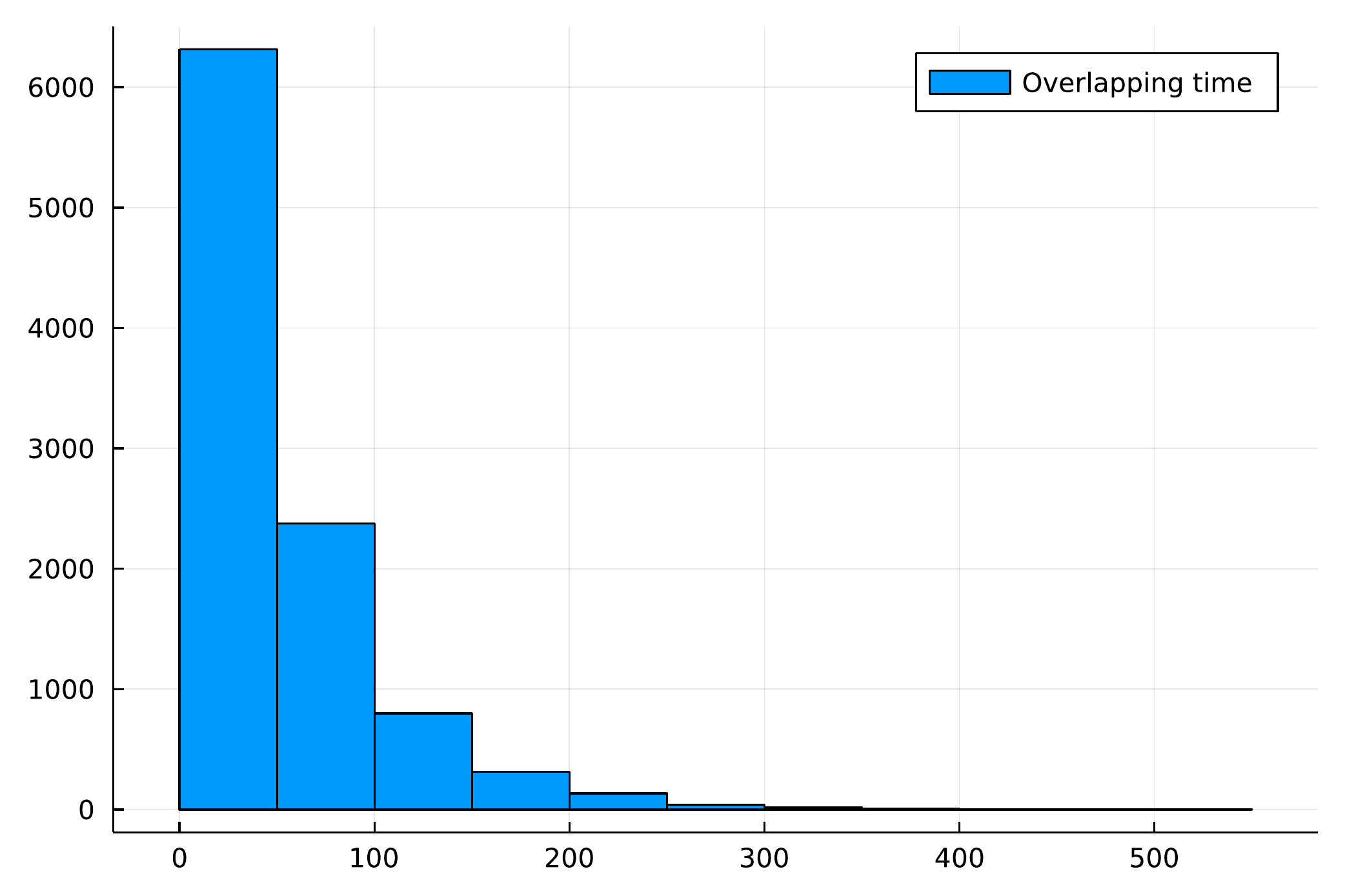}}\subfigure[\(\tau = 5\)] {
		\includegraphics[width = .4\textwidth]{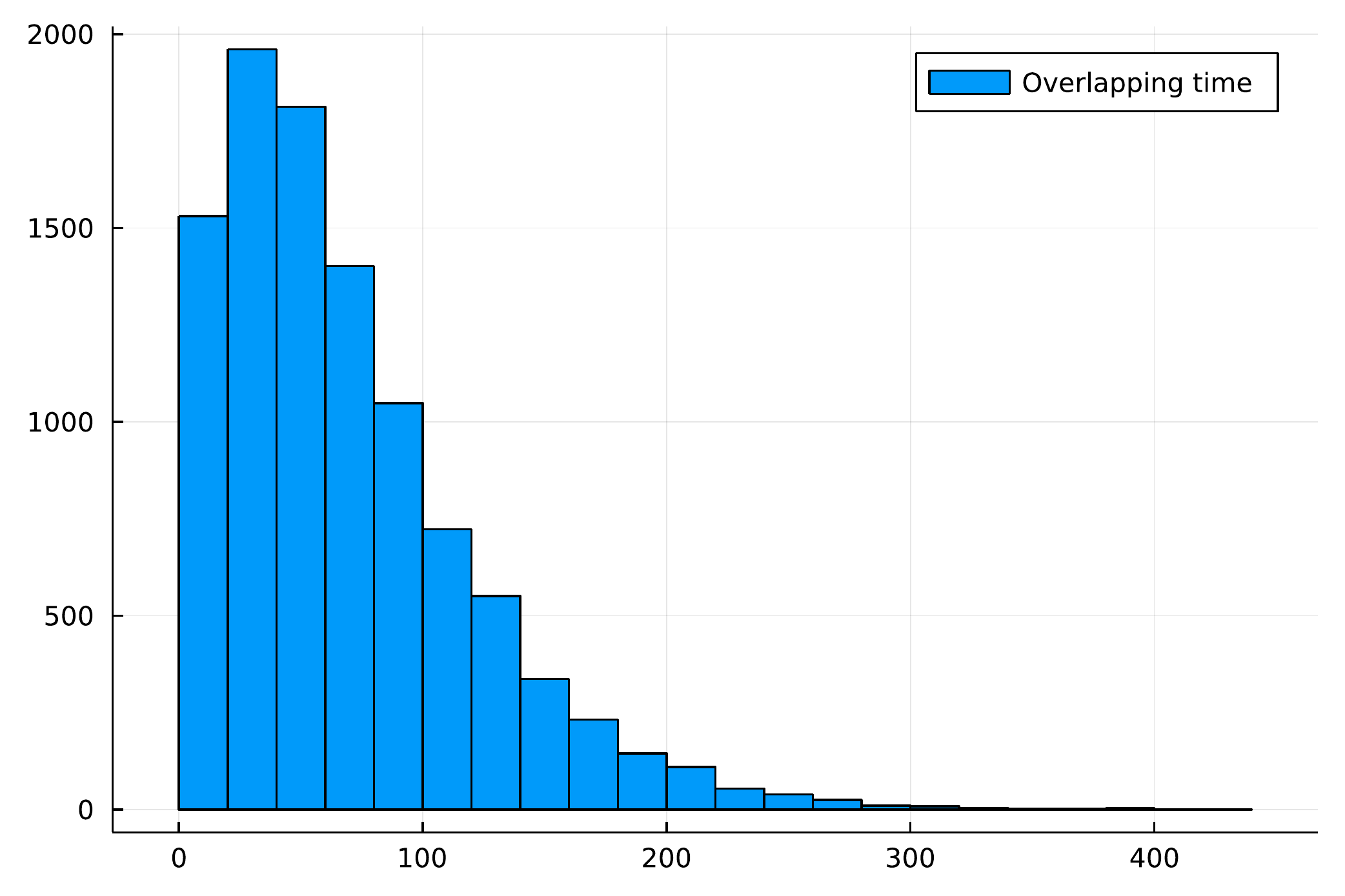}}
		\subfigure[\(\tau = 7\)] {
		\includegraphics[width = .4\textwidth]{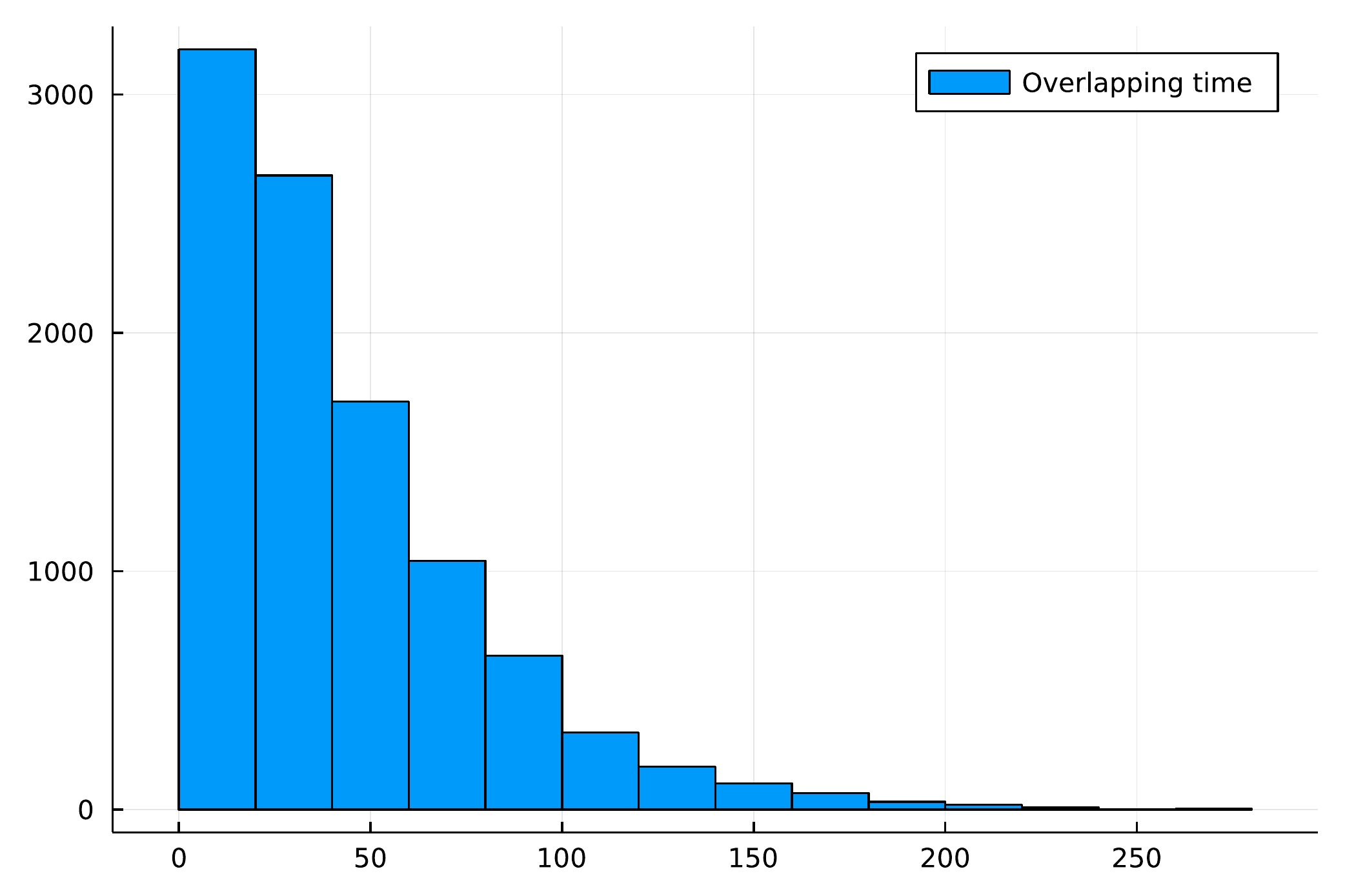}}\subfigure[\(\tau = 9\)] {
		\includegraphics[width = .4\textwidth]{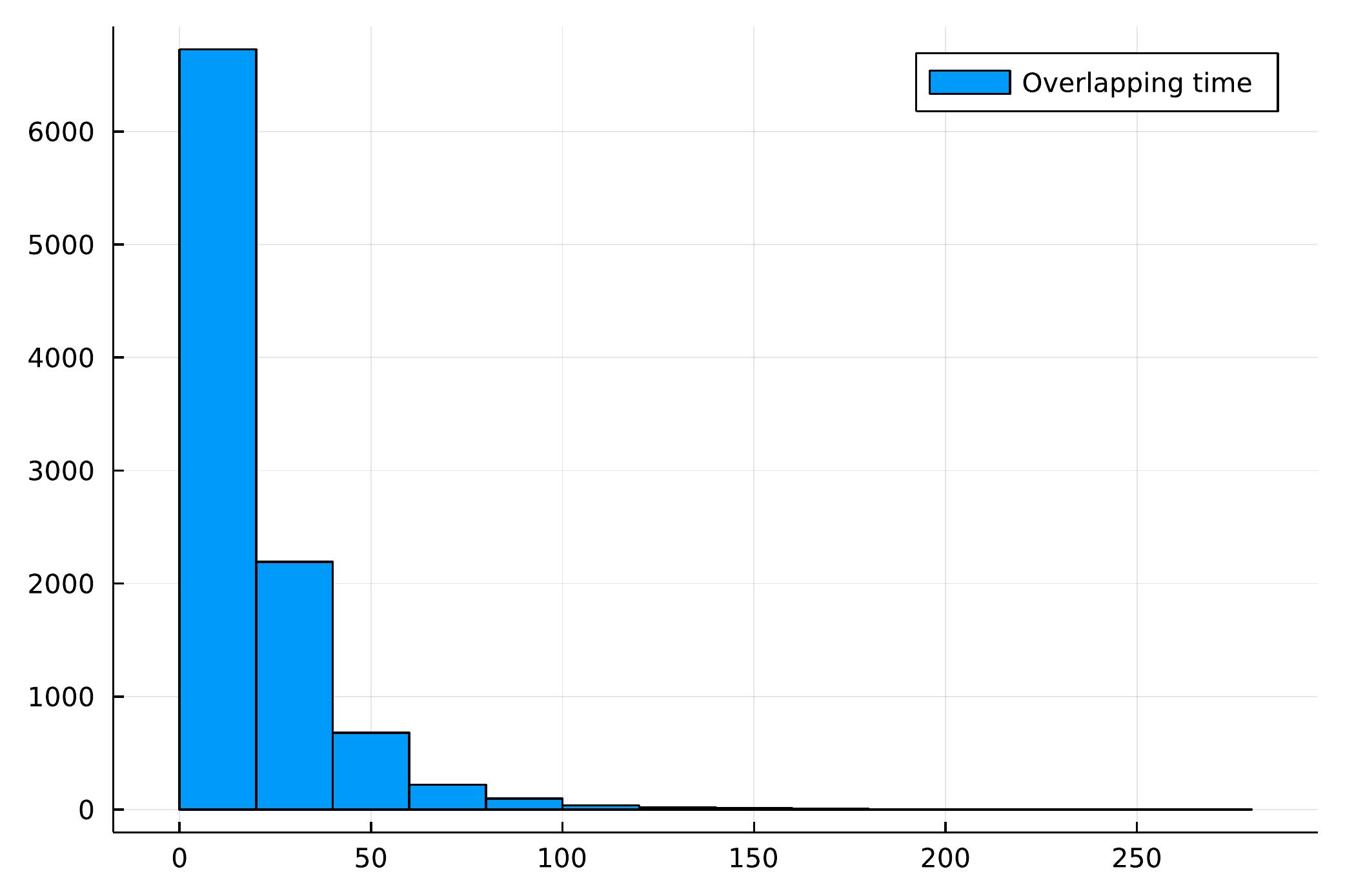}}
\end{center}
\caption{Histogram of overlapping times} \label{fig:histogram}
\end{figure*}

\section{Overlapping time for infinite-server queues} \label{sec:infinite}
This section extends the discussion from multi-server queues to the infinite-server queue. There are two notable things for the infinite-server queue compared with multi-server queues. First, infinite-server queues do not have the issues caused by lingering around critically loaded points. It means that we do not need the adjusted model. In fact, the adjusted model is identical to the fluid model in infinite-server queues. Second, in this case, we can derive a closed-form expression of the number of customers in the fluid limit.

Similar to the multi-server case, we construct a sample path for an \(M_t/M/\infty\) queue as follows:
\[X(t) = X(0) + Y_1\left(\int_0^t \lambda(s) ds\right) - Y_2\left(\int_0^t \mu X(s) ds\right).\]

We define a sequence of stochastic processes indexed by \(\eta\) as follows:
\[X^{\eta}(t) = X^{\eta}(0) + Y_1\left(\int_0^t \eta \lambda(s) ds\right) - Y_2\left(\int_0^t \eta \mu X^{\eta}(s) ds\right).\]

\begin{lemma}
\[\lim_{\eta\rightarrow \infty}\frac{X^{\eta}(t)}{\eta} = x(t)\text{ a.s.,}\]
where \(x(t)\) is the solution to the following ODE:
\[\frac{d}{dt}x(t) = \lambda(t) - \mu x(t) \text{ with } x(0) = \lim_{\eta\rightarrow \infty}\frac{X^{\eta}(0)}{\eta}.\]
\begin{proof}
This is a special case of the result in Kurtz \cite{Kurtz1978} and Mandelbaum et al. \cite{Mandelbaum1998}.
\end{proof}
\end{lemma}

If we assume an empty queue at the beginning, i.e., \(x(0) = 0\), we have the closed-form solution given by:
\[x(t) = e^{-\mu t}\left(v(t)-v(0)\right),\]
where \(v'(t)=\lambda(t)e^{\mu t}\).

Let us consider the following time-varying arrival rate showing the cyclic behavior of the arrival process:
\[\lambda(t) = \beta \sin(\alpha t) + \lambda.\]
Then, we can obtain the closed-form expression of \(v(t)\) as follows:
\begin{align}
    &v(t) = \frac{\beta}{\mu}\left(1+\frac{\alpha^2}{\mu^2}\right)^{-1}\left(\sin \alpha t -\frac{\alpha}{\mu}\cos \alpha t\right)e^{\mu t} + \frac{\lambda}{\mu}e^{\mu t}+ C.\nonumber\\
\end{align}
We then have
\begin{align}
    &v(t)-v(0)=\frac{\beta}{\mu}\left(1+\frac{\alpha^2}{\mu^2}\right)^{-1}\left(\sin \alpha t -\frac{\alpha}{\mu}\cos \alpha t\right)e^{\mu t} \nonumber \\
    &\quad\quad\quad\quad\quad+ \frac{\lambda}{\mu}e^{\mu t} + \frac{\alpha\beta}{\mu^2}\left(1+\frac{\alpha^2}{\mu^2}\right)^{-1}-\frac{\lambda}{\mu}\nonumber\\
    &\quad\quad\quad\quad= \frac{\beta}{\mu}\left(1+\frac{\alpha^2}{\mu^2}\right)^{-1}\left\{\left(\sin \alpha t -\frac{\alpha}{\mu}\cos \alpha t\right)e^{\mu t}\right.\nonumber\\
    &\quad\quad\quad\quad\quad\left.+\frac{\alpha}{\mu}\right\}+ \frac{\lambda}{\mu}\left(e^{\mu t}-1\right).\nonumber
\end{align}
We can thus obtain the closed-form expression of $x(t)$ as follows:
\begin{align}
    &x(t) = e^{-\mu t}\left(v(t)-v(0)\right)\nonumber\\
    &\;\;\;\;\;= \frac{\beta}{\mu}\left(1+\frac{\alpha^2}{\mu^2}\right)^{-1}\left\{\sin \alpha t -\frac{\alpha}{\mu}\cos \alpha t +\frac{\alpha}{\mu}e^{-\mu t}\right\}\nonumber\\
    &\;\;\;\;\;\quad+ \frac{\lambda}{\mu}\left(1-e^{-\mu t}\right).\label{eqn:x_inf}
\end{align}

\begin{figure}[t!]\label{fig:one_cycle}
	  \begin{center}
	  \includegraphics[width = .45\textwidth]{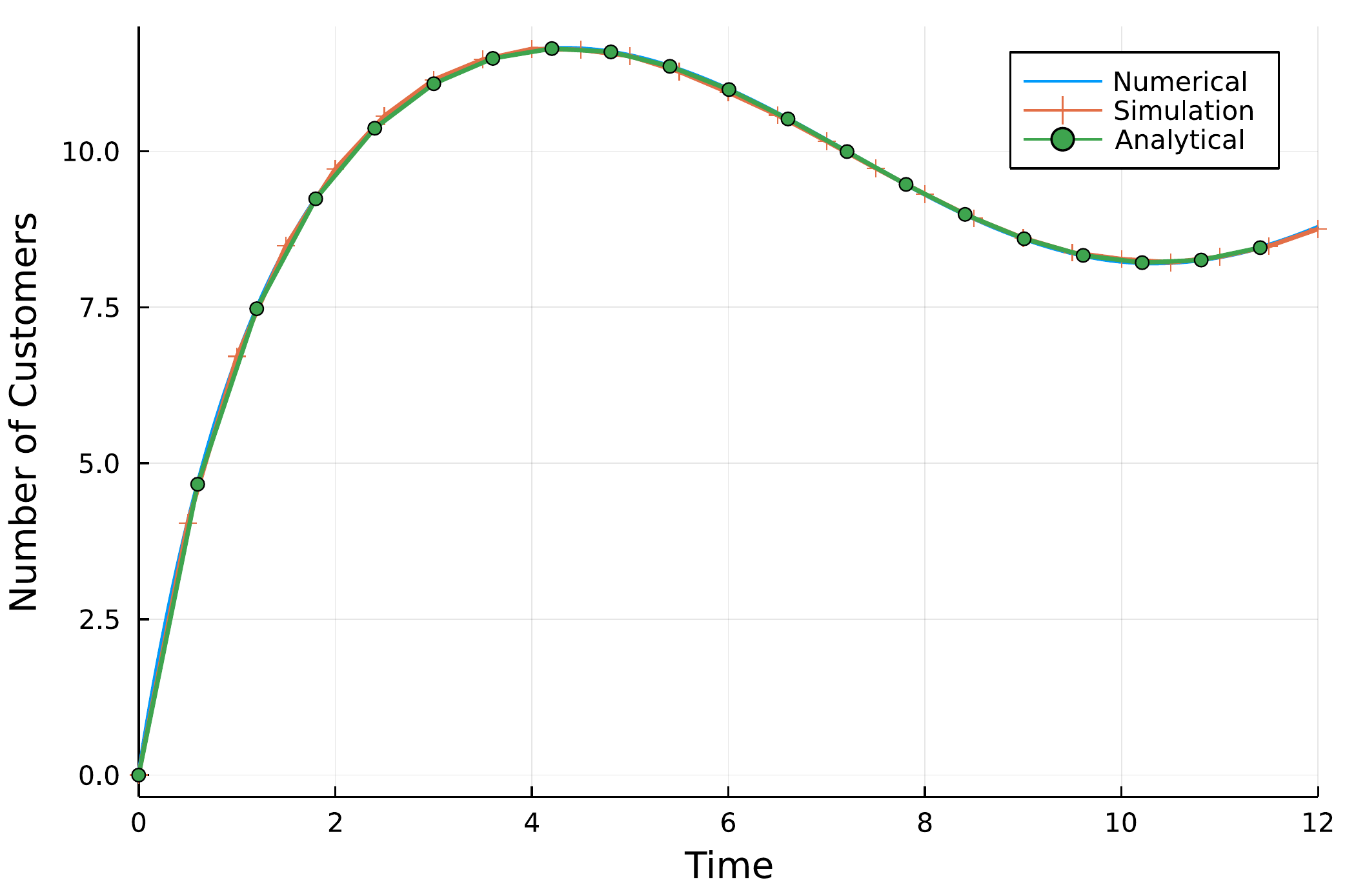}
		  \caption{Mean dynamics of the system (analytical vs. numerical vs. simulation)} \label{fig:comp_inf}
	  \end{center}
\end{figure}

In Figure~\ref{fig:comp_inf}, we plot the number of customers in an $M_t/M/\infty$ queue with the rate function (\ref{eqn:lambda_t_mmn}) and parameters \(\alpha = 0.5\), \(\beta = 2.0\), and \(\lambda = 10.0\). We can observe that the fluid limit accurately approximates the mean behavior of the \(M_t/M/\infty\) queue.

Since a virtual customer gets service immediately upon arrival in the infinite-server system, we do not need to consider the \(Z(\tau,t)\) process to compute the overlapping time. The customer overlaps with other customers until time \(\tau+1/\mu\) on average. We then approximate the mean overlapping time of a virtual customer arriving at time \(\tau\) as:

\begin{align}
& E[O_{\tau}] \approx\int_{t=\tau}^{\tau+\frac{1}{\mu}}x(t)dt \nonumber\\
 & =\frac{\beta}{\mu}\bigg(1+\frac{\alpha^{2}}{\mu^{2}}\bigg)^{-1}\bigg\{-\frac{1}{\alpha}\big(\cos\alpha(\tau+\frac{1}{\mu})-\cos\alpha\tau\big) \nonumber\\
 & -\frac{1}{\mu}\big(\sin\alpha(\tau+\frac{1}{\mu})-\sin\alpha t\big)-\frac{\alpha(e^{-\mu(\tau+\frac{1}{\mu})}-e^{-\mu\tau})}{\mu^{2}}\bigg\} \nonumber\\
 & +\frac{\lambda}{\mu^{2}}\big(1+e^{-\mu(\tau+\frac{1}{\mu})}-e^{-\mu\tau}\big). \label{eqn:closed_form_approx}
\end{align}

Table~\ref{tbl:anal_vs_sim_inf} provides a numerical study for estimating the mean overlapping times for infinite-server queues with the same setting as Figure~\ref{fig:comp_inf}. The ``Analytical'' column is from the closed-form solution, the ``Numerical'' column provides the solutions by numerically solving the ODE, and the ``Simulation'' column contains the result of 10,000 simulation runs.

As expected, the fluid limit accurately approximates the mean overlapping times of time-varying infinite-server queues. Also, the numerical solution verifies the correctness of the closed-form expression in \eqref{eqn:closed_form_approx}.

\begin{table}[h!]
\centering
\caption{Overlapping times: analytical (closed-form) vs numerical  vs simulation}\label{tbl:anal_vs_sim_inf}
\begin{tabular}{c|c|c|c}
\hline
Arrival time  & Analytical & Numerical & Simulation\\ \hline
 3.0 & 11.41 & 11.41 & 11.42 \\ 
 5.0 & 11.30 & 11.30 & 10.96 \\ 
 7.0 & 9.74 & 9.74 & 9.40  \\
 9.0 & 8.39  & 8.39 & 8.38  \\ \hline 
\end{tabular}
\end{table}
\section{Conclusion}\label{sec:conclusion}
This paper proposes a fluid-limit-based approach for approximating overlapping times of virtual customers arriving at arbitrary time points in a multi-server queueing system. The standard fluid limit estimation becomes inaccurate when the system lingers around the critically loaded point. We fix this issue by using an adjusted fluid model, and the numerical study shows that it significantly improves the estimation accuracy.

We extend our discussion to the infinite-server setting, where the fluid limit no longer has the lingering problem. We derive the closed-form expression of the overlapping time, and the numerical results verify the correctness of our approach.

There can be a few future extensions of this work. First, the distribution of overlapping time is useful in decision-making regarding risk management. We do not derive the distribution of overlapping times in this work since deriving the variance of the limit process could be challenging. Second, we mainly focus our discussion on a Markovian setting in this paper. We will extend this result to more general settings using other limiting theories in our future work.

\bibliographystyle{abbrv}
\bibliography{ref.bib}







\end{document}